\documentclass[11pt]{amsart}

\usepackage[utf8]{inputenc}

\usepackage{amssymb}
\usepackage{upref}
\RequirePackage{doi}
\usepackage{hyperref}
\usepackage[capitalize]{cleveref}
\usepackage{mathtools}
\usepackage{tikz-cd}
\usepackage{graphicx, caption}
\usepackage{comment,enumerate}
\usepackage{thmtools}
\usepackage{thm-restate}

\usepackage{adjustbox,lipsum}

\numberwithin{equation}{section}

\makeatletter \renewcommand{\p@enumii}{} \makeatother

\newtheorem{thm}{Theorem}[section]
\newtheorem{lemma}[thm]{Lemma}
\newtheorem{prop}[thm]{Proposition}
\newtheorem{cor}[thm]{Corollary}

\newtheorem{construction}[thm]{Construction}

\crefformat{lem}{Lemma~#2\textup{#1}#3}
\Crefname{lem}{Lemma}{Lemmas}
\crefname{lem}{lemma}{lemmas}
\crefmultiformat{lem}{Lemmas~#2\textup{#1}#3}{ and~#2\textup{#1}#3}{, #2\textup{#1}#3}{ and~#2\textup{#1}#3}
\crefformat{cor}{Corollary~#2\textup{#1}#3}
\Crefname{cor}{Corollary}{Corollaries}
\crefname{cor}{corollary}{corollaries}
\crefformat{prop}{Proposition~#2\textup{#1}#3}
\Crefname{prop}{Proposition}{Propositions}
\crefname{prop}{proposition}{propositions}
\crefmultiformat{prop}{Propositions~#2\textup{#1}#3}{ and~#2\textup{#1}#3}{, #2\textup{#1}#3}{ and~#2\textup{#1}#3}
\crefformat{thm}{Theorem~#2\textup{#1}#3}
\Crefname{thm}{Theorem}{Theorems}
\crefname{thm}{theorem}{theorems}

\theoremstyle{definition}

\newtheorem{definition}[thm]{Definition}

\newtheorem*{notation*}{Notation}

\theoremstyle{remark}

\crefmultiformat{rem}{Remarks~#2\textup{#1}#3}{ and~#2\textup{#1}#3}{, #2\textup{#1}#3}{ and~#2\textup{#1}#3}
\crefformat{rem}{Remark~#2\textup{#1}#3}
\crefformat{rems}{Remark~#2\textup{#1}#3}
\Crefname{rem}{Remark}{Remarks}
\crefname{rem}{remark}{remarks}

\crefformat{figure}{Figure~#2\textup{#1}#3}

\newcounter{case}
\newenvironment{case}[1][\unskip]{\refstepcounter{case}\bf
\medskip \noindent Case \thecase\ #1.\ \it}{\unskip\upshape}
\numberwithin{case}{thm}
\renewcommand{\thecase}{\arabic{case}}
\crefformat{case}{Case~#2\textup{#1}#3}
\Crefname{case}{Case}{Cases}
\crefname{case}{case}{cases}

\newcounter{subcase}
\newenvironment{subcase}[1][\unskip]{\refstepcounter{subcase}\bf
\medskip \noindent Subcase \thesubcase\ #1.\ \it}{\unskip\upshape}
\numberwithin{subcase}{case}
\crefformat{subcase}{Subcase~#2\textup{#1}#3}
\crefname{subcase}{subcase}{subcases}
\Crefname{subcase}{Subcase}{Subcases}

\newcounter{subsubcase}

\numberwithin{subsubcase}{subcase}
\crefformat{subsubcase}{Subsubcase~#2\textup{#1}#3}
\crefname{subsubcase}{subsubcase}{subsubcases}
\Crefname{subsubcase}{Subsubcase}{Subsubcases}

\usepackage{color}
\theoremstyle{definition}
\newtheorem{aid}{}
\numberwithin{aid}{section}
\newcommand{\oldaid}{}
\let\oldaid=\aid
\renewcommand{\aid}{\bigbreak\oldaid}
\newcommand{\oldendaid}{}
\let\oldendaid=\endaid
\renewcommand{\endaid}{\oldendaid\bigskip\hrule width\textwidth \bigbreak}
\crefformat{aid}{Note~#2\textup{#1}#3}

\let\oldenumerate=\enumerate
\def\enumerate{\oldenumerate \itemsep=\smallskipamount}
\let\olditemize=\itemize
\def\itemize{\olditemize \itemsep=\smallskipamount}

\newcommand{\floor}[1]{\lfloor #1 \rfloor}
\DeclareMathOperator{\Aut}{Aut}

\DeclareMathOperator{\id}{id}
\DeclareMathOperator{\Sym}{Sym}

\DeclareMathOperator{\supp}{supp}
\DeclareMathOperator{\DihGrp}{D}

\newcommand{\UTFC}{\mathcal{UTFC}}
\newcommand{\DU}{\mathcal{UTFD}}

\newcommand{\TW}{\mathcal{T}}
\newcommand{\CX}{\mathcal{CX}}
\newcommand{\U}{\mathcal{U}}
\newcommand{\CA}{\mathcal{CA}}
\newcommand{\UTF}{\mathcal{UTF}}

\newcommand{\pref}[1]{\textup(\ref{#1}\textup)}

\newcommand{\fullcref}[2]{\cref{#1}\pref{#1-#2}}

\makeatletter
\newcommand{\noprelistbreak}{\smallskip\@nobreaktrue\nopagebreak} 
\makeatother

\begin{document}

%\title{Stability - asymptotics}
\title[Asymptotics of stable and unstable graphs]{Asymptotical results on the automorphism groups  of canonical double cover of graphs}
\author[B. Bychawski]{Bart\l omiej Bychawski }
\address{Faculty of Mathematics and Computer Science, Jagiellonian University, Krakow, Poland and Faculty of Mathematics, University of Cambridge, UK} 
\email{bartlomiej.bychawski@student.uj.edu.pl}
\author[A. Hujdurovi\'c]{Ademir Hujdurovi\'c }
\address{University of Primorska, Faculty of Mathematics, Natural Sciences and Information Technologies, Glagolja\v ska 8, 6000 Koper, Slovenia}
\email{ademir.hujdurovic@upr.si}
\author[\DJ. Mitrovi\'c]{\DJ or\dj e Mitrovi\'c }
\address{Department of Mathematics, University of Auckland, Private Bag 92019, Auckland 1142, New Zealand}
\email{dorde.mitrovic@auckland.ac.nz}
\author[G. Verret]{Gabriel Verret}\address{Department of Mathematics, University of Auckland, Private Bag 92019, Auckland 1142, New Zealand}
\email{g.verret@auckland.ac.nz}

%\date{\today}

\begin{abstract}
 The canonical double cover $BX$ of a graph $X$ is the direct product $X\times K_2$. A graph $X$ is stable if every automorphism of $BX$ is induced by automorphisms of $X$ and $K_2$, and unstable otherwise. 
%It is nontrivially unstable if it is unstable, connected, nonbipartite, and no two distinct vertices have the same set of neighbours.
We prove several asymptotic results on stable and unstable graphs. In particular, we show that the automorphism group of the canonical double cover of almost every graph has order $2$, and consequently that almost all graphs are stable. We also prove that the dominant sources of instability are having twins and being disconnected, in that order, and that graphs arising from a known construction are the next step in this hierarchy.
%that almost all unstable graph have twins, and that almost all nontrivially unstable graph arise from a known construction. Finally, we show that the automorphism group of the canonical double cover of almost every nontrivially unstable graph is isomorphic to the dihedral group of order $8$.

\end{abstract}

\maketitle

\begin{quotation}
\end{quotation}
\begin{quotation}
\end{quotation}

%\tableofcontents

\section{Introduction}
All graphs in this paper are finite and simple. Given a graph $X$, its automorphism group is denoted by $\Aut(X)$. The  \textit{canonical bipartite double cover} (also called the \textit{bipartite double cover} or the \textit{Kronecker cover}) of $X$, denoted by $BX$, is the direct product $X \times K_2$ (where $K_2$ denotes the complete graph on two vertices). In other words, we have that 
\[V(BX)=V(X)\times \{0,1\}\] and \[E(BX)=\big\{\{(x,0),(y,1)\}: \{x,y\}\in E(X)\big\}.\] We denote by $\tau$ the \textit{canonical involution} on $V(BX)$ interchanging $(x,0)$ and $(x,1)$ for all $x\in V(X)$. We observe that $\tau\in \Aut(BX)$ and that $\Aut(X)\times \langle\tau\rangle \cong \Aut(X)\times \mathrm{S}_2$ canonically embeds into $\Aut(BX)$. The graph $X$ is \emph{stable} if $BX$ has no additional automorphism, that is, if 
\[\Aut(BX) = \Aut(X)\times \langle\tau\rangle.\]
Otherwise, $X$ is called \emph{unstable}.

Canonical double covers have proven to play an  important role in  algebraic graph theory and have been studied from a variety of perspectives. For example, their connections with symmetric $\{0,1\}$-matrices  were studied by Maru\v si\v c, Scapellato and Salvi ~\cite{maruvsivc1989characterization}, their regular embeddings by Nedela and \v Skoviera \cite{RegularEmbeddingsNedela}, their isomorphisms by Pacco and Scapellato \cite{Pacco} and  their symmetry properties in the context of $2$-fold covers by Feng, Kutnar, Malni\v c and Maru\v si\v c \cite{FengKutnar2fold}. 

Wilson gave several sufficient conditions for a graph to be unstable, which he then used to formulate conjectures characterising instability in four families of graphs: circulants, Rose Window graphs, toroidal graphs and generalized Petersen graphs~\cite{WilsonUnExpected}.  Wilson's conjectures were proven for generalized Petersen graphs by Qin, Xia, and Zhou~\cite{Qin21}, for toroidal graphs by Witte Morris~\cite{MorrisADAM}, and for Rose Window graphs by Ahanjideh, Kov\' acs, and Kutnar~\cite{AhanjidehKK}. 

A lot of work has been done on the stability of circulant graphs. Wilson's conjecture was shown to be false by  Qin, Xia and Zhou, who also proved that there are no nontrivially unstable arc-transitive circulants~\cite{QinXiaZhou}. It was proven by Fernandez and the second author that there are no nontrivially unstable circulants of odd order \cite{FernandezHujdurovicCirc}. Morris extended this result by proving that there are no nontrivially unstable Cayley graphs on abelian groups of odd order~\cite{morris2021automorphisms}. In \cite{HMMAuto}, several constructions of unstable circulants were presented, and unstable circulants of order twice a prime were classified, whereas in \cite{HMMAtMost7} unstable circulants with valency at most $7$ were classified. The connection between  stability of Cayley graphs and properties of Schur rings was studied in \cite{HKSchur}. Nevertheless, the full classification of unstable circulants remains an open problem. In ~\cite{BinzhouShashaZhishou_StableCayleyAsymptotics}, Xia, Zhang and Zheng proved that almost all Cayley graphs on abelian groups are stable. 

In this paper, we study the asymptotic behavior of the proportion of various properties related to stability of general graphs. Throughout the paper all asymptotic statements concern labeled graphs; in particular, when we say that almost all graphs in a family have a given property, we mean that the proportion of graphs  with that property among  all graphs in the family  with vertex-set $[n]\coloneqq \{1,\ldots,n\}$  tends to $1$ as $n\to\infty$. For a precise definition, see \cref{defn:AlmostAll}. 

We first show that instability is asymptotically rare. In fact, we prove the following stronger statement.

\begin{restatable}{thm}{zero}
\label{AlmostAllGraphsAreStable}
Almost all graphs $X$ satisfy $\Aut(BX)=\langle \tau\rangle$.
\end{restatable}

As a direct consequence, we have the following:

\begin{cor}
Almost all graphs are stable.
\end{cor}

It is not hard to see that graphs with \textit{twins} (nonadjacent distinct vertices with the same neighbourhoods) and disconnected graphs are unstable. We will prove that these are the dominant sources of instability, in that order.

\begin{restatable}{thm}{second}
\label{thm:unstable have twins}
Almost all unstable graphs have twins.
%\begin{enumerate}
%    \item ,\label{AlmostAllUnstableTwins}
%    \item are connected and\label{AlmostAllUnstableConnected}
%    \item are nonbipartite.\label{AlmostAllUnstableNonBipartite}
%\end{enumerate}
\end{restatable}

%As a consequence, almost all connected unstable graphs still have twins.
%We move further down the hierarchy and prove that, among twin-free graphs, being disconnected is the dominant source of unstability.

\begin{restatable}{thm}{third}
\label{thm:allmost all twin-free are disconnected}
Almost all unstable twin-free  graphs are disconnected.
\end{restatable}

Moving further down the asymptotic hierarchy, we consider twin-free connected unstable graphs and show that they almost all arise from
the following simple construction, which was first defined in \cite{SomeConditionsImplyingStability}.

\begin{construction}\label{constructionX(A,B)}
Let $X$ be a graph and let $A$ and $B$ be subsets of $V(X)$. Let $X(A,B)$ be the graph obtained from $X$ by adding four new vertices $a_1,a_2,b_1,b_2$ and two new edges $\{a_1,b_1\},\{a_2,b_2\}$ to $X$, and then joining $a_1$ and $a_2$ with every vertex in $A$ and $b_1$ and $b_2$ with every vertex in $B$.
\end{construction}

\begin{figure}[!htbp]
    \centering
    \includegraphics{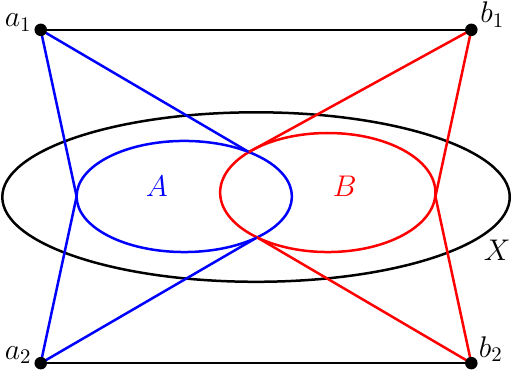}
    \caption{$X(A,B)$.}
    \label{fig:construction}
\end{figure}

%We prove that this construction accounts for almost all nontrivially unstable graphs. 

\begin{restatable}{thm}{fourth}
\label{thm:almost all nontrivially are XAB}
Almost all unstable  twin-free connected  graphs arise via Construction~\ref{constructionX(A,B)}.
\end{restatable}

%{\color{red}Note that each graph arising from Construction~\ref{constructionX(A,B)} admits $(a_1,a_2)(b_1, b_2)$  as an automorphism (see also \Cref{X(AB)Unstable}). We obtain the following corollary of Theorem~\ref{thm:almost all nontrivially are XAB}.  

%\begin{cor}
%    Almost all unstable twin-free connected graphs admit a double transposition as an automorphism.
%\end{cor}}

Finally, we determine the typical automorphism group of the canonical double cover of the graphs in \Cref{thm:almost all nontrivially are XAB}. Let $\DihGrp_4$ denote the dihedral group of order $8$.

\begin{restatable}{thm}{fifth}
\label{cor:typical-index-two-nontrivial}
Almost all unstable  twin-free connected  graphs $Y$ that arise via Construction~\ref{constructionX(A,B)} satisfy
$$\Aut(BY)=\left\langle((a_1,0),(a_2,0))((b_1,1),(b_2,1)),\tau\right\rangle\cong \DihGrp_4$$
for some distinct vertices $a_1,a_2,b_1,b_2$ of $Y$.
\end{restatable}

Some authors (for example Wilson in ~\cite{WilsonUnExpected}) define a graph to be \textit{nontrivially unstable} if it is unstable but twin-free, connected and nonbipartite. Note that nonbipartite is slightly different from the other two conditions in this definition: graphs with twins and disconnected graphs  are necessarily unstable, whereas a bipartite graph may still be stable. Moreover, we have already noted that, asymptotically speaking, graphs with twins and disconnected graphs are the dominant sources of instability. Being bipartite is much further down the list, see for example~\Cref{counting}. It is also lower than graphs arising from Construction~\ref{constructionX(A,B)}. For this reason, \Cref{thm:almost all nontrivially are XAB} and \Cref{cor:typical-index-two-nontrivial} immediately hold with ``unstable  twin-free connected'' replaced by ``nontrivially unstable''.

%and, although we do not prove this, we expect it is also dominated by similar constructions related to graphs with nontrivial automorphism of small support.

We first prove some important preliminary results in \Cref{sec:prelim} and then the main theorems in \Cref{sec:main}.

%In Section~2, we develop counting estimates for automorphisms of canonical double covers and prove Theorem~\ref{thm:almost all are stable}. In Section~3, we prove Theorem~\ref{thm:unstable have twins}. In Section~4 we analyse nontrivially unstable graphs and prove Theorems~\ref{thm:allmost all twin-free are disconnected}, \ref{thm:almost all nontrivially are XAB} and \ref{cor:typical-index-two-nontrivial}.

\section{Preliminaries}\label{sec:prelim}

Let \(X\) be a graph with vertex-set $V$. For $v\in V$, the \textit{neighbourhood} $N_X(v)$ of $v$ in $X$ is the set of vertices of $X$ adjacent with $v$. 
Note that  
\begin{equation*}\label{Eq:NBX}
    N_{BX}(v,0) = N_X(v)\times \{1\} \text{ and } N_{BX}(v,1) = N_X(v)\times \{0\}.
\end{equation*} 
%Thus, $N_{BX}(v,0)$ is uniquely determined by $N_X(v)$ and vice-versa. 
Let \(g_1,g_2\in \Sym(V)\). We denote by \((g_1,g_2)\) the permutation of \(V(BX)=V\times\{0,1\}\) defined by 
\[
 (v,0)^{(g_1,g_2)}=(v^{g_1},0) \quad\text{and}\quad (v,1)^{(g_1,g_2)}=(v^{g_2},1).
\]

%Recall that $\tau$ is the automorphism of $BX$ interchanging $(v,0)$ and $(v,1)$ for all $v\in V$. 
%Note that $\tau$ interchanges $V\times\{0\}$ and $V\times\{1\}$, whereas  $(g_1,g_2)$ with $g_1,g_2\in\Sym(V)$ preserves them.

%In particular, if \(X\) is connected, nonbipartite and twin-free, then \(X\) is nontrivially unstable if and only if there exist \(g_1,g_2\in \Sym(V)\), with \(g_1\ne g_2\), such that $(g_1,g_2)\in \Aut(BX).$

Given a group $G\leq\Sym(V)$, the \textit{support} of $G$, denoted  $\supp(G)$, is the set of elements of $V$ not fixed by $G$.

\begin{lemma}\label{lem:split}
   Let $X$ be a graph with vertex-set $V$, let $v\in V$ and let $g_1,g_2\in \Sym(V)$ such that $(g_1,g_2) \in \Aut(BX)$. Let $G\coloneqq\langle g_1,g_2\rangle$, $W\coloneqq  V\setminus\supp(G)$ and $g\in G$. Then 
   \begin{enumerate}
   \item $N_X(v^{g_1})=\left( N_{X}(v) \right)^{g_2} \text{ and } N_X(v^{g_2})=\left( N_{X}(v)\right)^{g_1}.$ \label{EqNeighbourhoods}
   \item $N_{X}(v^g) \cap W=N_{X}(v) \cap W.$ \label{EqNeighbourhoodsInFixed}
    \item $v$ and ${v}^{g_2^{-1}g_1}$ are nonadjacent in $X$.\label{LEM: no edges between v and image}
   \end{enumerate}
If moreover $X$ is twin-free, then
\begin{enumerate}\setcounter{enumi}{3}
\item If one of $g_1$ or $g_2$ is the identity, they both are. \label{lem:(g,1)ImpliesTwins}
\item $g_1$ and $g_2$ have the same order. \label{LEM: when X twin-free pi0 pi1 and iota are iso}

\item   If $g_1\neq g_2$, then $|\supp(G)|\geq 4.$ \label{lem:support at least 4}
\end{enumerate}
\end{lemma}

\begin{proof}
\begin{enumerate}
    \item 
We have
\begin{align*}
        N_X(v^{g_1})\times \{1\} &= N_{BX}(v^{g_1},0) = N_{BX}((v,0)^{(g_1,g_2)})  \\
        &= \left(N_{BX}(v,0)\right)^{(g_1,g_2)} = \left(N_X(v)\times\{1\}\right)^{(g_1,g_2)}  \\
        &= N_{X}(v)^{g_2}\times \{1\}.
\end{align*}
It follows that  $N_X(v^{g_1})=N_{X}(v)^{g_2}$, and an analogous argument yields $N_X(v^{g_2})=N_{X}(v)^{g_1}$.  
\item  From (\ref{EqNeighbourhoods}), we get
$$N_X(v^{g_1})\cap W=N_{X}(v)^{g_2}\cap W=(N_{X}(v)\cap W)^{g_2}=N_{X}(v)\cap W.$$
A symmetrical argument yields that $N_X(v^{g_2})\cap W=N_{X}(v)\cap W$ and thus $N_X(v^g)\cap W=N_{X}(v)\cap W$.
    \item Note that $(v^{g_2^{-1}},0)$ and $(v^{g_2^{-1}},1)$ are nonadjacent in $BX$. Consequently, since $(g_1,g_2) \in \Aut(BX)$, we have that \[{(v^{g_2^{-1}},0)}^{(g_1,g_2)} = (v^{g_2^{-1}g_1},0)\] and \[{(v^{g_2^{-1}},1)}^{(g_1,g_2)} = (v,1)\] are also nonadjacent in $BX$. It follows that $v^{g_2^{-1}g_1}$ and $v$ are nonadjacent in $X$, as desired.
\end{enumerate}

\begin{enumerate}\setcounter{enumi}{3}
\item    Without loss of generality, suppose that $g_2=\id$. Applying (\ref{EqNeighbourhoods}), we get $N_{X}(v) = N_{X}(v^{g_1})$. Since $X$ is twin-free, we have that $v=v^{g_1}$ and, since $v$ is arbitrary, we have $g_1=\id$, as required.

\item Assume for contradiction that $\mathrm{ord}(g_1) \neq \mathrm{ord}(g_2)$. Without loss of generality, we may assume that $\mathrm{ord}(g_1)$ does not divide $\mathrm{ord}(g_2)$ and thus $g_1^{\mathrm{ord}(g_2)} \neq \id,$ and
    $$
    \left(g_1^{\mathrm{ord}(g_2)},\id\right)
    =
    {(g_1,g_2)}^{\mathrm{ord}(g_2)}
    \in
    \Aut(BX),
    $$
    contradicting (\ref{lem:(g,1)ImpliesTwins}).
\item     Suppose for a contradiction that $|\supp(G)|\leq 3$. 
By (\ref{LEM: when X twin-free pi0 pi1 and iota are iso}), $g_1$ and $g_2$ have the same order. Since $g_1\neq g_2$, either $g_1$ and $g_2$ are distinct transpositions or $g_1$ is a $3$-cycle and $g_2 = g_1^{-1}$. In both cases, $g_2^{-1}g_1$ is a cycle of length $3$ on $\supp(G)$. By (\ref{LEM: no edges between v and image}), it follows that the three vertices contained in $\supp(G)$ are  nonadjacent in $X$. Let $v$ such that $v\neq v^{g_1}$, in particular, $v\in \supp(G)$.  By (\ref{EqNeighbourhoodsInFixed}), we have $N_{X}(v^{g_1}) \cap W=N_{X}(v) \cap W$ and thus $v$ and $v^{g_1}$ are twins in $X$, contradicting the assumption that $X$ is twin-free.
\end{enumerate}
\end{proof}

\begin{lemma}\label{basic}
A connected nonbipartite graph $X$ with vertex-set $V$ is unstable if and only if there exist \(g_1,g_2\in \Sym(V)\), with \(g_1\ne g_2\), such
that $(g_1,g_2)\in \Aut(BX).$
\end{lemma}
\begin{proof}
Since \(X\) is connected and nonbipartite,  it follows by Weichsel’s Theorem (see \cite[Theorem 5.9]{ProductHandbook}) that \(BX\) is connected hence $\{V\times\{0\},V\times\{1\}\}$ is the unique bipartition of $BX$ and thus every automorphism of $BX$ preserves it.

Suppose that $X$ is unstable. Then $\Aut(X)\times \langle\tau\rangle<\Aut(BX)$. By the previous paragraph, there exist \(g_1,g_2\in \Sym(V)\) with $(g_1,g_2)\in \Aut(BX)$ and $(g_1,g_2)$ not of the form  \((g,g)\) with
\(g\in \Aut(X)\).  If $g_1=g_2$, then Lemma~\ref{lem:split}(\ref{EqNeighbourhoods}) implies that $g_1\in \Aut(X)$, a contradiction, so we have $g_1\neq g_2$, as required. The converse is obvious.
\end{proof}

\begin{lemma}\label{lem:Orbits}
Let $g_1,g_2\in \Sym(V)$, let \(G\coloneqq\langle g_1,g_2\rangle\) and let $G^*\coloneqq\langle (g_1,g_2),\tau\rangle\leq\Sym(V\times \{0,1\})$. Then  the orbits of $G^*$ are exactly the sets of the form $O\times \{0,1\}$ with $O$ an orbit of $G$.
\end{lemma}
\begin{proof}
    Since $\tau \in G^*$, an orbit of $G^*$ contains $(x,0)$ if and only if it contains $(x,1)$. Using this and the fact that $\tau (g_1,g_2)\tau = (g_2,g_1) \in G^*$, the conclusion easily follows.
\end{proof}

\begin{cor} \label{COR: if suppG has order 4 then Y=X(A,B)}
     Let $Y$ be a twin-free graph with vertex-set $V$ and let $g_1,g_2 \in \mathrm{Sym}(V)$ such that $g_1 \neq g_2$ and $(g_1,g_2) \in \Aut(BY)$. Let $G\coloneqq\langle g_1,g_2\rangle$ and let $X$ be the subgraph of $Y$ induced on $V\setminus\supp(G)$. If $|\supp(G)|=4$,  then there exist subsets $A,B \subseteq V(X)$ such that $Y=X(A,B)$.
\end{cor}
\begin{proof}
Since $|\supp(G)|=4$, either $G$ has two orbits of size $2$ or one orbit of size $4$.
  
\begin{case} 
$G$ has an  orbit of size $4$.
\end{case}

 Let ${O}$ be this orbit, so $W\coloneqq V\backslash {O}=V\setminus\supp(G)$. Let $Z$ be the subgraph of $Y$ induced on ${O}$. Then $BZ$ is a subgraph of $BY$ and, by Lemma~\ref{lem:Orbits}, $\langle(g_1,g_2),\tau \rangle\leq \Aut(BY)$ acts transitively on $O\times \{0,1\}=V(BZ)$. In particular, both $BZ$ and $Z$ are regular graphs. 

Let $a\in{O}$ and let $A\coloneqq N_Y(a)\cap W$. By Lemma~\ref{lem:split}(\ref{EqNeighbourhoodsInFixed}), $A$ does not depend on the choice of $a$. In particular, since $Y$ is twin-free, so is $Z$.  Since $g_1\neq g_2$, $\supp(g_1,g_2)\subseteq O\times \{0,1\}=V(BZ)$ and $(g_1,g_2) \in \Aut(BY)$, it follows that $Z$ is unstable. The only regular twin-free graphs on four vertices are  $2K_2$ and $K_4$, but $K_4$ is stable hence $Z\cong 2K_2$. Hence   $Y=X(A,A)$,  as required.

 \begin{case} 
 $G$ has two orbits of size $2$.
\end{case}

Let  $\{a_1,a_2\}$ and $\{b_1,b_2\}$ be distinct  orbits of size $2$ of  $G$, so $W=V\backslash\{a_1,a_2,b_1,b_2\}$. By Lemma~\ref{lem:split}(\ref{LEM: when X twin-free pi0 pi1 and iota are iso}), $g_1$ and $g_2$ have the same order. Since $g_1 \neq g_2$, up to relabeling, we have that either
    \[\text{ $g_1 = (a_1,a_2)(b_1,b_2)$ and $g_2 = (a_1,a_2)$}\]
    or
    \[\text{$g_1 = (a_1,a_2)$ and $g_2 = (b_1,b_2)$}.\]

\begin{subcase}
 $g_1 = (a_1,a_2)$ and $g_2 = (b_1,b_2)$.
\end{subcase}

Let  $W\coloneqq V\setminus\supp(G)=V\setminus\{a_1,a_2,b_1,b_2\}$. Since $g_2^{-1}g_1 = (a_1,a_2)(b_1,b_2)$, 
   Lemma~\ref{lem:split}(\ref{LEM: no edges between v and image}) implies that $a_1$ is nonadjacent to $a_2$ in $Y$ and  $b_1$ is nonadjacent to $b_2$ in $Y$. By Lemma~\ref{lem:split}(\ref{EqNeighbourhoodsInFixed}), we have
    \[
    A \coloneqq N_{Y}(a_1) \cap W= N_{Y}(a_1^{g_1}) \cap W=N_{Y}(a_2) \cap W
    \]
    and
    \[
    B \coloneqq N_{Y}(b_1) \cap W=N_{Y}(b_2^{g_2}) \cap W=N_{Y}(b_2) \cap W.
    \]
    
Since $Y$ is twin-free, $a_1$ and $a_2$ are not twins, so they are each adjacent to exactly one of $b_1$ or $b_1$, and vice versa. Without loss of generality, let $a_1$ be adjacent to $b_1$ and nonadjacent to $b_2$ and $a_2$ be adjacent to $b_2$ and nonadjacent to $b_1$. It now follows that $Y=X(A,B)$, as required.
    
\begin{subcase}\label{LEM 5.3 bad case}
 $g_1 = (a_1,a_2)(b_1,b_2)$ and $g_2 = (a_1,a_2)$.
\end{subcase}

   Let $g\coloneqq(b_1,b_2)$  and note that    $(g,g)=(g_1g_2,g_2g_1)= (g_1,g_2)(g_1,g_2)^\tau\in \Aut(BY)$. By Lemma~\ref{lem:split}(\ref{EqNeighbourhoods}) applied with $v=b_1$, we have 
   $$N_Y(b_2)=N_Y(b_1^{g_1})=N_{Y}(b_1)^{g_2}$$
   and 
      $$N_Y(b_2)=N_Y(b_1^{g})=N_{Y}(b_1)^{g}.$$
   Let $V_a\coloneqq V\setminus \{a_1,a_2\}$ and $V_b\coloneqq V\setminus \{b_1,b_2\}$. We now have 
  $$N_Y(b_2)\cap V_a=N_{Y}(b_1)^{g_2}\cap V_a=N_{Y}(b_1)\cap V_a$$ 
  and
    $$N_Y(b_2)\cap V_b=N_{Y}(b_1)^{g}\cap V_b=N_{Y}(b_1)\cap V_b.$$
    Since $V_a\cup V_b=V$, it follows that $N_Y(b_2)=N_{Y}(b_1)$, contradicting the fact that $Y$ is twin-free.
\end{proof}

Finally, we recall some basic properties of Construction~\ref{constructionX(A,B)}.

\begin{prop}[{\cite[Proposition 5.2]{SomeConditionsImplyingStability}}]\label{X(AB)Unstable}
Let $X$ be a graph with vertex-set $V$ and let $A,B\subseteq V$. Then
\begin{enumerate}
\item $(a_1,a_2)(b_1,b_2)\in\Aut(X(A,B))$,
\item $\gamma\coloneqq((a_1,0),(a_2,0))((b_1,1),(b_2,1))$
is an automorphism of the double cover of $X(A,B)$,
\item we have $\langle \gamma,\tau\rangle\cong\DihGrp_4$.
\end{enumerate}
In particular, $X(A,B)$ is unstable.  Moreover, if $X$ is connected and twin-free, and at least one of $A$ and $B$ is nonempty, then $X(A,B)$ is also connected and twin-free.
\end{prop}
\begin{proof}
The fact that $(a_1,a_2)(b_1,b_2)\in\Aut(X(A,B))$ follows immediately from Construction~\ref{constructionX(A,B)} (and, in particular, Figure~\ref{fig:construction}). By \cite[Proposition 5.2]{SomeConditionsImplyingStability}, $\gamma$ is an automorphism of the double cover of $X(A,B)$. Both $\gamma$ and $\tau$ are involutions, while their product has order $4$, hence, 
$\langle\gamma,\tau\rangle \cong \DihGrp_4$, as required. The last part is immediate.
\end{proof}

For fixed \(g_1, g_2 \in \Sym(V)\), we give upper bounds on the number of graphs \(X\) with vertex-set $V$ such that \((g_1,g_2)\in \Aut(BX)\).

\begin{lemma}\label{Lem:supportsize}
Let $n$ be a positive integer, let $V$ be a set of order $n$, let $g_1,g_2\in \Sym(V)$ and let
\(G\coloneqq\langle g_1,g_2\rangle\).
Let $\mathcal{X}$ be the set of graphs $X$   with vertex-set $V$ such that $(g_1,g_2) \in \Aut(BX)$. Let $s$ be the size of the support of $G$ and let the nontrivial orbits of $G$ have sizes $a_1\geq a_2 \geq \cdots \geq a_k\geq 2$. Then

\begin{enumerate}
\item  \label{lem:g1g2bound} 
 $|\mathcal{X}|\leq 2^{(n-1)+(n-1-a_1)+(n-1-a_1-a_2)+\cdots+(n-1-a_1-\cdots-a_{k-1})+\binom{n-s}{2}}.$
 \item \label{m/2Lemma} If $m\geq 0$ is an even integer with $m\leq s$, then 
\[|\mathcal{X}|\leq
2^{\binom{n}{2}-\frac{m}{2}\left(n-\frac{m}{2}-1\right)}.
\]
\end{enumerate}
\end{lemma}

\begin{proof}
Let \(O_1,\ldots,O_k\) be the nontrivial orbits of \(G\), with $|O_i|=a_i\geq 2$ for all \(i\in \{1,\ldots,k\}\). Then
\[
s= \sum_{i=1}^k a_i.
\]

Let $X\in\mathcal{X}$ and let $v_i\in O_i$. By Lemma~\ref{lem:Orbits}, $O_i\times \{0,1\}$ is an orbit of $\langle (g_1,g_2),\tau\rangle$. It follows that $N_X(v_i)$ uniquely determines $N_{BX}(w)$ for every $w\in O_i\times \{0,1\}$, which in turn uniquely determines $N_X(u)$ for every $u\in O_i$. 
%Hence, specifying a neighbourhood of $v_i\in O_i$ in $X$, completely determines the neighbourhoods of all other vertices in $O_i$. 

There are at most $2^{n-1}$ possibilities for the set $N_X(v_1)$. By the previous paragraph, this choice then determines $N_X(u_1)$ for every $u_1\in O_1$. It follows that $N_X(v_2)\cap O_1$ is also determined, which leaves at most $2^{n-a_1-1}$ possibilities for $N_X(v_2)$, which in turn determines $N_X(u_2)$ for every $u_2\in O_2$. We repeat this procedure for all $i\in \{1,\ldots,k\}$ obtaining that there are at most $2^{n-a_1-\ldots-a_{k-1}-1}$ possibilities for $N_X(v_k)$, hence determining $N_X(u_k)$ for every $u_k\in O_k$. It remains to determine the adjacencies amongst the vertices of $X$ fixed by $G$, that is, those not contained in $O_1\cup\cdots\cup O_k$. There are $n-s$ such vertices, so at most $2^{\binom{n-s}{2}}$ possibilities, which concludes the proof of (\ref{lem:g1g2bound}).

We now prove \pref{m/2Lemma}. Write $m=2t$. We call the vertices of $O_i$ other than $v_i$ \emph{non-representative}. We order the vertices in the following way: first $v_1$, then non-representative vertices in $O_1$, then $v_2$, then non-representative vertices in $O_2$, etc., with fixed vertices at the end. We choose the neighbourhood of the vertices in order. In the usual count of the total number of labeled graphs 
$$2^{\binom{n}{2}}=2^{n-1}\cdot 2^{n-2}\cdots 2^{0},$$
the vertex in the $i$-th position contributes $2^{n-i}$. In this case,  the adjacency of a non-representative vertex is already determined, as in the proof of (\ref{lem:g1g2bound}).  Therefore, if the vertex in the $i$-th position is non-representative, then we lose a factor of  $2^{n-i}$ compared to the usual count.

Since $a_i\geq 2$, the $j$-th non-representative vertex has position at most $2j$. Therefore, the factor lost because of
the $j$-th non-representative vertex is at least
\[
2^{n-2j}.
\]
There are $s-k$ non-representative vertices. Since $a_i\geq 2$, we have $k\leq s/2$ and therefore
\[
s-k \geq s/2\geq m/2=t.
\]
Hence there are at least $t$ non-representative vertices. Consequently, the total loss compared from $2^{\binom{n}{2}}$ is at least
$$2^{n-2}\cdot 2^{n-4}\cdots 2^{n-2t}=2^{t(n-t-1)},$$
as required.
\end{proof}

\begin{lemma}\label{Lem:supportsize5}
Let $n$ be a positive integer and let $V$ be a set of order $n$. The number of graphs $X$ with vertex-set $V$ for which there exist $g_1,g_2\in \Sym(V)$ such that $|\supp(\langle g_1,g_2\rangle)|\geq 5$  and $(g_1,g_2) \in \Aut(BX)$ is $o\left(2^{\frac{n^2-5n+4}{2}}\right)$.
\end{lemma}
\begin{proof}
Let $g_1$ and $g_2$ be as in the statement, and let $G=\langle g_1,g_2 \rangle$.
We consider two cases.

\begin{case}
    $|\supp(G)|=5$.
\end{case} 
The number of  pairs $(g_1,g_2)$  in this case is at most
\[
        \binom n5(5!)^2.
\]
The support of $G$ is either a single orbit of size
$5$, or a union of two orbits of sizes $3$ and $2$.
For each fixed pair $(g_1,g_2)$, Lemma~\ref{Lem:supportsize}(\ref{lem:g1g2bound}) gives at most
\[
       2^{n-1} 2^{\binom{n-5}{2}} = 2^{\frac{n^2-9n+28}{2}}
\]
graphs $X$ when $G$ has an orbit of size $5$, and at most
\[
    2^{n-1} 2^{n-4} 2^{\binom{n-5}{2}} = 2^{\frac{n^2-7n+20}{2}}
\]
when $G$ has an orbit of size $3$ and one of size $2$. The
total contribution of this case is thus at most
\begin{equation*}
\label{eq:SuppG=5}
        \binom n5(5!)^2 
        \left(
            2^{\frac{n^2-9n+28}{2}}+ 2^{\frac{n^2-7n+20}{2}}
        \right)
        =
        o\left(2^{\frac{n^2-5n+4}{2}}\right).
\end{equation*}

\begin{case}
    $|\supp(G)|\geq 6$.
\end{case}
Let $x\coloneqq 2\left\lceil 4\log_2 n\right\rceil$. Note that  $x$ is a positive even integer.

\begin{subcase}
   $ x\geq |\supp(G)|\geq 6$.
\end{subcase}
The number of  pairs $(g_1,g_2)$  in this case is at most
\[
        \binom nx (x!)^2
        \le
         n^{2x}.
\]
For each fixed pair $(g_1,g_2)$, Lemma~\ref{Lem:supportsize}(\ref{m/2Lemma})
applied with $m=6$ gives at most
\[
        2^{\binom n2-3(n-4)}
        =
        2^{\frac{n^2-7n+24}{2}}
\]
graphs $X$ satisfying $(g_1,g_2)\in \Aut(BX)$. The
total contribution of this case is thus at most
\begin{equation*}
\label{eq:x>=SuppG>=6}
        n^{2x}2^{\frac{n^2-7n+24}{2}}=2^{\frac{n^2-7n+24}{2}+2x\log_2 n}=o\left(2^{\frac{n^2-5n+4}{2}}\right).
\end{equation*}
%Dividing by $2^{\binom{n-4}{2}+2n-8}$ gives
%\[
%         n^{2x}2^{-n+10}
%        =
%        2^{O((\log_2 n)^2)-n}
%        \longrightarrow 0 \text{ as } n\rightarrow \infty.
%\]
%Therefore, this contribution is
%\[
%o\left(2^{\binom{n-4}{2}+2n-8}\right).
%\]

\begin{subcase}
    $|\supp(G)|>x$.
\end{subcase}
The number of  pairs $(g_1,g_2)$  in this case is at most
\[
        (n!)^2\le n^{2n}.
\]
For each fixed pair $(g_1,g_2)$, Lemma~\ref{Lem:supportsize}(\ref{m/2Lemma})
applied with $m=x$ gives at most
\[
        2^{\binom n2-\frac x2\left(n-\frac x2-1\right)}.
\]
graphs $X$ satisfying $(g_1,g_2)\in \Aut(BX)$. The total contribution of this case is thus at most
\[
        n^{2n}2^{\binom n2-\frac x2\left(n-\frac x2-1\right)}=2^{\binom n2-\frac x2\left(n-\frac x2-1\right)+2n\log_2 n}.
\]
Dividing by $2^{\frac{n^2-5n+4}{2}}$ gives
\[
 2^{2n-2-\frac x2\left(n-\frac x2-1\right)+2n\log_2 n }.
\]
Since $x=2\lceil 4\log_2 n\rceil$, the above exponent tends to $-\infty$ when $n\rightarrow \infty$. Therefore this ratio is
\begin{equation*}
\label{eq:SuppG>x}
2^{2n-2-\frac x2\left(n-\frac x2-1\right)+2n\log_2 n } = o\left(1\right),
\end{equation*}   
which completes the proof.
\end{proof}

\begin{lemma}\label{lem:TwoDoubleTranspositions}
    The number of graphs with vertex-set $[n]$ admitting two distinct double transpositions as automorphisms is at most \[n^82^{\frac{n^2-7n+48}{2}}.\]
\end{lemma}
\begin{proof}
Note that $\Sym(n)$ contains $3\binom{n}{4}$ double transpositions, so there are at most $n^8$ choices for a pair $\alpha,\beta\in\Sym(n)$ of distinct double transpositions.  We estimate the number of graphs with vertex-set $[n]$ having $\alpha$ and $\beta$ as automorphisms. Let $G:=\langle \alpha,\beta\rangle$,  $U \coloneqq \supp(G)$  and  $r\coloneqq |U|$. Clearly, $4\leq r\leq 8$. 

Each of $\alpha$ and $\beta$ is a product of two disjoint transpositions. Since $\alpha$ and $\beta$ are distinct, these gives us at least $3$ distinct transpositions and it follows that $G$ has at most $r-3$ orbits on $U$.

There are $2^{\binom{n-r}{2}}$ choices for the edges within $[n]\setminus U$,  and at most $2^{\binom{r}{2}}$ choices for the edges within $U$. Finally, if $v\in [n]\setminus U$ and $u\in U$, then $v$ is adjacent to $u$ if and only if it is adjacent to every element of the $G$-orbit of $u$. It follows that there are at most $2^{r-3}$ choices for the edges between $v$ and $U$, and thus $2^{(r-3)(n-r)}$ choices for edges between $[n]\setminus U$ and $U$. This gives us an upper bound of 

\[
2^{\binom{n-r}{2}+(r-3)(n-r)+\binom r2}
=
2^{\frac{n^2-7n+6r}{2}}
\leq
2^{\frac{n^2-7n+48}{2}}\]
possible graphs.
\end{proof}

  \section{Main results}  \label{sec:main}

\begin{definition}
    \label{defn:AlmostAll}
    Let $\mathcal{F}$ and $\mathcal{H}$  be  families of graphs and let $\mathcal{G}_n$ be the family of all graphs with vertex-set $[n]$. We say that ``\textit{almost all graphs in $\mathcal{F}$ are in  $\mathcal{H}$}'' if
    \[
    \lim_{n\rightarrow \infty}
    \frac
    {|\mathcal{F}\cap\mathcal{G}_n\cap \mathcal{H}|
          }
        {
            |\mathcal{F}\cap \mathcal{G}_n| 
        } 
    =
    1.
    \]
    If the above holds with $\mathcal{F}=\bigcup_{n\in\mathbb{N}}\mathcal{G}_n$, then we say that ``almost all graphs are in  $\mathcal{H}$''.
\end{definition} 

Recall that \(f(n) = o(g(n))\) if
\[
\lim_{n \to \infty} \frac{f(n)}{g(n)} = 0.
\]

\begin{lemma}\label{counting}
Let $n$ be a positive integer. Then 
    \begin{enumerate}
        \item \label{counting-disconnected}The number of disconnected graphs with vertex-set $[n]$
    is $(1-o(1))n2^{\binom{n-1}{2}}$. In particular, almost all graphs are connected.
        \item \label{counting-bipartite} The number of bipartite graphs with vertex-set $[n]$ is at most $2^{\frac{n^2}{4} + n}$. In particular, almost all graphs are nonbipartite.
        \item \label{counting-twins} The number of graphs with vertex-set $[n]$ with twins is
        $(1-o(1))\binom{n}{2} 2^{\binom{n-1}{2}}$. In particular, almost all graphs are twin-free.
    \end{enumerate}
\end{lemma}
\begin{proof}
\mbox{}

\pref{counting-disconnected}: See for example \cite[Example II.15, p.138]{AnalyticComb_Flajolet}.

\pref{counting-bipartite}:
    To count the bipartite graphs with vertex-set $[n]$, we consider all possible bipartitions of $[n]$ and then all possible edge sets for each of those bipartitions. Let $A$ and $B$ form a partition of $[n]$ with $0< |A|\leq |B| < n$. Writing $k\coloneqq |A|$, there are $\binom{n}{k}$ choices for $A$ and  $k(n-k)$ possible edges between $A$ and $B$. Hence the following is  an upper bound on the number of bipartite graphs with vertex-set $[n]$:

    \[\sum_{k=1}^{\floor{\frac{n}{2}}}\binom{n}{k}2^{k(n-k)} \leq 2^{\frac{n^2}{4}}\sum_{k=1}^{\floor{\frac{n}{2}}}\binom{n}{k}\leq 2^{\frac{n^2}{4}}\sum_{k=0}^{n}\binom{n}{k}\leq 2^{\frac{n^2}{4} + n}.\]

\pref{counting-twins}:
Let $\TW_n$ denote the set of graphs with vertex-set $[n]$ having twin
vertices, and, for distinct $i,j\in[n]$, let $\TW_n(i,j)$ denote the set
of graphs with vertex-set $[n]$ in which $i$ and $j$ are twins. Clearly,
\[
    \TW_n=\bigcup_{\{i,j\}\in\binom{[n]}{2}}\TW_n(i,j).
\]
By the first two Bonferroni inequalities (that is, the first two truncations of the inclusion–exclusion formula), we obtain
\begin{equation}\label{eq:Bonferroni}
    \sum_{\{i,j\}\in\binom{[n]}{2}}|\TW_n(i,j)|
    -
    \sum_{\substack{
        \{\{i,j\},\{k,\ell\}\}\subseteq \binom{[n]}{2}\\
        \{i,j\}\neq\{k,\ell\}
    }}
    |\TW_n(i,j)\cap\TW_n(k,\ell)|
    \leq |\TW_n|
    \leq
    \sum_{\{i,j\}\in\binom{[n]}{2}}|\TW_n(i,j)|.
\end{equation}

Let $i,j\in[n]$ be distinct. A graph with vertex-set $[n]$ in which
$i$ and $j$ are twins is uniquely determined by its subgraph induced on
$[n]\setminus\{j\}$. Hence,
\[
    |\TW_n(i,j)|=2^{\binom{n-1}{2}}.
\]
and 
\begin{equation}\label{eq:FirstEstimate}
    \sum_{{\{i,j\}\in\binom{[n]}{2}}}|\TW_n(i,j)|
    =
    \binom{n}{2}2^{\binom{n-1}{2}}.
\end{equation}
%The upper bound in \pref{eq:Bonferroni} therefore gives
%\[
%    |\TW_n|
%    \leq
%    \sum_{i,j\in [n],i\neq j}|\TW_n(i,j)|
%    =
%    \binom{n}{2}2^{\binom{n-1}{2}}.
%\]

There are at most $n^4$ choices for $(i,j,k,\ell)$ and, depending on whether $\{i,j\}\cap\{k,\ell\}=\varnothing$ or not, we either get a set of three distinct twins, or two disjoint sets of disjoint twins. Either way, we get exactly two ``redundant'' vertices, that is vertices whose neighbourhood is determined by their twins. It follows that, for fixed $(i,j,k,\ell)$, we have $$|\TW_n(i,j)\cap\TW_n(k,\ell)|=2^{\binom{n-2}{2}}.$$

This implies that \[
    \sum_{\substack{
        \{\{i,j\},\{k,\ell\}\}\subseteq \binom{[n]}{2}\\
        \{i,j\}\neq\{k,\ell\}
    }}
    |\TW_n(i,j)\cap\TW_n(k,\ell)|\leq n^42^{\binom{n-2}{2}}.
\]
but 
\[    \lim_{n\to\infty}
    \frac{
        n^42^{\binom{n-2}{2}}
    }{
        \binom{n}{2}2^{\binom{n-1}{2}}
    }
    =
    \lim_{n\to\infty}
    \frac{n^4}{\binom{n}{2}}\,2^{-(n-2)}
    =
    0.\]

Together with \pref{eq:Bonferroni} and \pref{eq:FirstEstimate}, this yields
\[
    |\TW_n|
    =
    (1-o(1))
    \binom{n}{2}2^{\binom{n-1}{2}},
\]
as desired.

\end{proof}

\zero*
%\begin{thm}\label{AlmostAllGraphsAreStable}
%Almost all graphs $X$ satisfy $\Aut(BX)=\langle \tau\rangle$.
%\end{thm}

\begin{proof}
Combining \fullcref{counting}{disconnected} and \fullcref{counting}{bipartite}, we conclude that almost all graphs are connected and nonbipartite. Hence, without loss of generality, we may assume that $X$ is connected and nonbipartite. As noted earlier, this implies that $BX$ is connected and thus $\{V(X)\times \{0\},V(X)\times \{1\}\}$ is the  unique bipartition of $BX$. It follows that, if $\Aut(BX)> \langle \tau \rangle$, then there exist $g_1,g_2\in \Sym(V)$ such that $(\id,\id)\neq (g_1,g_2)\in \Aut(BX)$.  Let $G\coloneqq\langle g_1,g_2\rangle$. Note that $|\supp(G)|\geq 2$.

We will give an upper bound on the number of graphs $X$ with vertex-set $[n]$  such that $(g_1,g_2)\in \Aut(BX)$. We closely follow the ideas of the proof of \cite[Lemma 2.3.2]{GodsilRoyle}. Let $x\coloneqq 4\lceil{(\log_2 n)^2}\rceil$. Note that $x$ is a positive even integer. 

\begin{case} 
$|\supp(G)|\leq x$.
\end{case}
The number of  pairs $(g_1,g_2)$  in this case is at most
\[
        \binom{n}{x} (x!)^2\leq n^{2x}.
\]
For each fixed pair $(g_1,g_2)$, Lemma~\ref{Lem:supportsize}(\ref{m/2Lemma}) applied with $m=2$ gives at most
\[
2^{\binom{n}{2}-n+2}
\]
graphs $X$ satisfying $(g_1,g_2)\in \Aut(BX)$. The
total contribution of this case is thus at most
\[
         n^{2x}2^{\binom{n}{2}-n+2}= 2^{\binom{n}{2}-n+2+2x\log_2 n}=o\left(2^{\binom{n}{2}}\right).
\]
\begin{case} 
$|\supp(G)|> x$.
\end{case}
The number of  pairs $(g_1,g_2)$  in this case is at most
\[
        (n!)^2 \leq n^{2n}.
\]
For each fixed pair $(g_1,g_2)$, Lemma~\ref{Lem:supportsize}(\ref{m/2Lemma}) applied with $m=x$ gives at most
\[
2^{\binom{n}{2}-\frac{x}{2}(n-\frac{x}{2}-1)}
\]
graphs $X$ satisfying $(g_1,g_2)\in \Aut(BX)$. The
total contribution of this case is thus again at most
\[
         n^{2n}2^{\binom{n}{2}-\frac{x}{2}(n-\frac{x}{2}-1)} = 2^{\binom{n}{2}-\frac{x}{2}(n-\frac{x}{2}-1)+2n\log_2(n)}=o\left(2^{\binom{n}{2}}\right),
\]
which concludes the proof.
\end{proof}

We now give an upper bound on the number of unstable twin-free connected   graphs which will be used in the proofs of Theorems~\ref{thm:unstable have twins} and \ref{thm:allmost all twin-free are disconnected}.
\begin{lemma}\label{lem:connected-twin-free-unstable}
The number of  unstable twin-free connected graphs
with vertex-set $[n]$ is
$o\left(2^{\binom{n-1}{2}}\right)$.
\end{lemma}

\begin{proof}
Using Lemma~\ref{counting}(\ref{counting-bipartite}) it is not dificult to see that the
number of bipartite graphs is
$o\left(2^{\binom{n-1}{2}}\right)$.
It therefore suffices to count unstable twin-free connected  nonbipartite
graphs. For each such graph $X$, by Lemma~\ref{basic} there exist
$g_1,g_2\in\Sym(n)$ with $g_1\neq g_2$ and
$(g_1,g_2)\in\Aut(BX)$. Let $G=\langle g_1,g_2\rangle$.
By Lemma~\ref{lem:split}(\ref{lem:support at least 4}),
we have $|\supp(G)|\geq4$. There are at most $\binom n4(4!)^2$ pairs $(g_1,g_2)$
with $|\supp(G)|=4$. For each fixed pair,
Lemma~\ref{Lem:supportsize}(\ref{m/2Lemma}), applied
with $m=4$, gives at most
$2^{\binom n2-2n+6}$ graphs. Thus their total number
is at most
\[
\binom n4(4!)^2\,2^{\binom n2-2n+6}
=o\left(2^{\binom{n-1}{2}}\right).
\]
By Lemma~\ref{Lem:supportsize5}, the number of graphs
admitting such a pair with $|\supp(G)|\geq5$ is
\[
o\left(2^{\frac{n^2-5n+4}{2}}\right)
=o\left(2^{\binom{n-1}{2}}\right).
\]
The result follows.
\end{proof}

\second*

\begin{proof}
Let $\U_n$ denote the set of unstable graphs
with vertex-set $[n]$, and let $\TW_n$ denote the set of
graphs on $[n]$ containing a pair of twins. Note that $\TW_n\subseteq \U_n$. 
%We shall prove that $|\U_n\setminus \TW_n| = o(|\TW_n|)$, from which the conclusion follows. 
By Lemma~\ref{counting}(\ref{counting-twins}), we have that
\begin{equation}\label{TwinCount}
        |\TW_n| = (1-o(1))\binom n2 2^{\binom{n-1}{2}}.
\end{equation}

We now estimate $|\U_n\setminus \TW_n|$, that is the number of unstable twin-free  graphs. First,  by Lemma~\ref{counting}(\ref{counting-disconnected}), the
number of disconnected graphs on $[n]$ is \[(1-o(1))n2^{\binom{n-1}{2}}.\] Comparing to (\ref{TwinCount}), we see this is $o(|\TW_n|)$. 
It remains to count unstable twin-free connected graphs. By
Lemma~\ref{lem:connected-twin-free-unstable}, their number is
\[
o\left(2^{\binom{n-1}{2}}\right)=o(|\TW_n|).
\]
Combining the estimates, we obtain
\[
|\U_n\setminus\TW_n|=o(|\TW_n|),
\]
and the result follows.
\end{proof}
%and (\ref{AlmostAllUnstableTwins}) follows. On the other hand, we saw near the start of the proof that the number of disconnected graphs and the number of bipartite graphs are both $o(|\TW_n|)$. Since every graph admitting twins is unstable, they are both $o(|\U_n|)$ (\ref{AlmostAllUnstableConnected}) as well and (\ref{AlmostAllUnstableNonBipartite}) immediately follow.

\begin{cor}
    The number of unstable graphs with vertex set $[n]$ is $(1-o(1))\binom{n}{2} 2^{\binom{n-1}{2}}$.
\end{cor}
\begin{proof}
    This follows directly from Theorem~\ref{thm:unstable have twins}, Lemma~\ref{counting}(3) and the fact that every graph admitting twins is unstable.
\end{proof}

\third*
\begin{proof}
Let $\UTF_n$ denote the set of unstable twin-free graphs with vertex-set $[n]$ and let $\DU_n$ denote the set of (unstable) twin-free disconnected graphs with vertex-set $[n]$. Clearly, $\DU_n \subseteq \UTF_n$. Observe that starting with a twin-free connected  graph with vertex-set $[n-1]$ and adding  an isolated vertex labelled $n$, we obtain a graph belonging to $\DU_n$. 
By Lemma \ref{counting}, it follows that $|\DU_n| \geq (1-o(1))2^{\binom{n-1}{2}}$. On the other hand, observe that $\UTF_n \setminus \DU_n$ consists of unstable  twin-free connected  graphs. By Lemma~\ref{lem:connected-twin-free-unstable} it follows that $|\UTF_n \setminus \DU_n|=o\left(2^{\binom{n-1}{2}}\right)=o(|\DU_n|)$, concluding the proof.
\end{proof}

%\begin{cor}\label{cor:almost all unstable are connected}
%    Almost all unstable graphs are connected.
%\end{cor}
%\begin{proof}
%    As was observed in the proof of Theorem~\ref{thm:unstable have twins} that the number of disconnected graphs is negligible compared to the number of graphs admitting twins. Since every graph admitting twins is unstable, it follows that disconnected graphs are negligible among unstable graphs. 
%\end{proof}

%\begin{cor}\label{cor:almost all connected have twins}
%Almost all connected unstable graphs have twins.
%\end{cor}

%\begin{proof}
%Since almost all unstable graphs are connected (\Cref{cor:almost all unstable are connected}), and almost all unstable graphs have twins (\Cref{thm:unstable have twins}), it follows that almost almost all connected unstable graphs have twins.
%\end{proof}

\begin{definition}\label{DefArise}
We say that a graph $Y$ \emph{arises via Construction~\ref{constructionX(A,B)}} if there exist distinct $a_1,a_2,b_1,b_2\in V(Y)$ and $A,B\subseteq V(Y)\setminus \{a_1,a_2,b_1,b_2\}$ such that $Y=X(A,B)$, where $X$ is the subgraph of $Y$ induced on $V(Y)\setminus \{a_1,a_2,b_1,b_2\}$.
\end{definition}
In the context of \Cref{DefArise}, we will allow the vertices $a_1,a_2,b_1,b_2$ of $Y$ to have other labels, such as elements of $[n]$.

Let $\CX_n$ be the set of graphs with vertex-set $[n]$ that arise via Construction~\ref{constructionX(A,B)}. Let $\UTFC_n$ denote the set of unstable twin-free connected graphs with vertex-set $[n]$. 

\begin{prop}\label{ConstructionBounds}
We have 
%\begin{enumerate}
 %   \item 
 $|\CX_n\cap \UTFC_n|\geq (1-o(1))2^{\frac{n^2-5n+4}{2}}$.
%\label{ConstructionLower}
%    \item $|\CX_n|\leq (1-o(1)) n^42^{\frac{n^2-5n+4}{2}}.$ {\color{red} we don't need this one anymore?}\label{ConstructionUpper}
%\end{enumerate}

\end{prop}
\begin{proof}
%We first prove the lower bound on $|\CX_n\cap \UTFC_n|$. 
Choose $X$ a connected twin-free graph with vertex-set $[n-4]$. By Lemma~\ref{counting}, there are 
$$(1-o(1))2^{\binom{n-4}{2}}$$ 
choices for $X$. Now, choose  $A,B\subseteq [n-4]$,
not both empty. There are 
$$2^{2(n-4)}-1$$ 
choices for $A$ and $B$. By Proposition~\ref{X(AB)Unstable}, $X(A,B)$ is an unstable twin-free connected graph. Let $Y$ be the graph obtained from $X(A,B)$ by relabelling $a_1,a_2,b_1,b_2$ with $n-3,n-2,n-1,n$, respectively. Clearly, $Y\in\CX_n\cap \UTFC_n$. Moreover,
$(X,A,B)$ can be recovered from $Y$ in the following way: $X$ is the subgraph of $Y$ induced by
$[n-4]$, $A=N(a_1)\cap [n-4]$ and $B=N(b_1)\cap [n-4]$. It follows that distinct $(X,A,B)$ yield distinct $Y$s and, in particular,
\begin{equation*}
        |\CX_n\cap \UTFC_n|
        \ge
        (1-o(1))2^{\binom{n-4}{2}+2n-8}=(1-o(1))2^{\frac{n^2-5n+4}{2}}, 
\end{equation*}
as required.

%We now prove the upper bound. Let $Y\in \CX_n$. By definition, there exist distinct $a_1,a_2,b_1,b_2\in [n]$ and $A,B\subseteq [n]\setminus \{a_1,a_2,b_1,b_2\}$ such that $Y=X(A,B)$, where $X$ is the subgraph of $Y$ induced on $[n]\setminus \{a_1,a_2,b_1,b_2\}$.
%There are $(1-o(1))n^4$ choices for $a_1,a_2,b_1,b_2$ and, once these are chosen, at most $2^{\binom{n-4}{2}}$ choices for $X$ and then at most $2^{n-4}$ choices for each of $A$ and $B$. Some of these choices may give the same $Y$, but we still get the upper bound
%\[
%|\CX_n|
%\leq
%(1-o(1))n^4
%2^{\binom{n-4}{2}+2(n-4)} = 
%(1-o(1))n^4
%2^{\frac{n^2-5n+4}{2}},
%\]
%as desired.
\end{proof}

\fourth*

\begin{proof}
We need to show 
\[
        |\UTFC_n\setminus \CX_n|=o\left(|\CX_n\cap \UTFC_n|\right).
\] 
By~\Cref{ConstructionBounds}, it suffices to show that 
\[
        |\UTFC_n\setminus \CX_n|=o\left(2^{\frac{n^2-5n+4}{2}}\right).
\]

Let $Y\in \UTFC_n\setminus \CX_n$. By \Cref{counting}(\ref{counting-bipartite}), the number of bipartite graphs is negligible, so we can assume that $Y$ is nonbipartite.  By Lemma~\ref{basic}, there exist $g_1,g_2\in \Sym(n)$ such that $g_1\neq g_2$ and $(g_1,g_2)\in \Aut(BY)$. Let $G=\langle g_1,g_2 \rangle$.
%Choose such a pair $(g_1,g_2)$, so that the size of the support of $G\coloneqq \langle g_1,g_2\rangle$ is smallest possible. 
By Lemma~\ref{lem:split}(\ref{lem:support at least 4}),  it follows that $|\supp(G)|\ge 4$. If $|\supp(G)|=4$, then by Corollary~\ref{COR: if suppG has order 4 then Y=X(A,B)} we have $Y \in \CX_n$, which is a contradiction. Hence, $|\supp(G)| \geq 5$ and the result follows from~\Cref{Lem:supportsize5}.
\end{proof}

\fifth*
\begin{proof}
Let $Y\in\CX_n\cap \UTFC_n$ and let $V:=V(Y)$. By definition, there exist distinct $a_1,a_2,b_1,b_2\in V$ and $A,B\subseteq V\setminus \{a_1,a_2,b_1,b_2\}$ such that $Y=X(A,B)$, where $X$ is the subgraph of $Y$ induced on $V\setminus \{a_1,a_2,b_1,b_2\}$. Let
\[
p\coloneqq (a_1,a_2)\in\Sym(V)
\qquad\text{and}\qquad
q\coloneqq (b_1,b_2)\in\Sym(V)
\]
Note that $p$ and $q$ have disjoint support and hence commute. It is clear that 
\[pq = (a_1,a_2)(b_1,b_2)\in \Aut(Y).\]
By Proposition~\ref{X(AB)Unstable}, we have that
\[
(p,q)=((a_1,0),(a_2,0))((b_1,1),(b_2,1))\in \Aut(BY),\]
and that
\[H\coloneqq\langle(p,q),\tau\rangle\cong \DihGrp_4.\]
Note that the kernel of the action of $H$ on the canonical bipartition  of $BY$ is
\[
K\coloneqq 
\left\{
(\id,\id),(p,q),(q,p),(pq,pq)
\right\}.\]
Recall that, from \Cref{ConstructionBounds}, we have that
\[|\CX_n\cap \UTFC_n|\geq (1-o(1))\left(2^{\frac{n^2-5n+4}{2}}\right).\] In order to establish the claim, it suffices to show that the subset $\CA_n$, consisting of graphs $Y\in\CX_n\cap \UTFC_n$ for which $\Aut(BY)>H$, is a negligible subset of $\CX_n$. In particular, we prove that
\[
|\CA_n|
=
o\left(2^{\frac{n^2-5n+4}{2}}\right).\]

Fix $Y\in\CA_n$. By \Cref{counting}(\ref{counting-bipartite}), the number of bipartite graphs is negligible, so we can assume that $Y$ is nonbipartite.  Since $Y$ is connected and nonbipartite, the assumption $\Aut(BY)>H$ implies that there exist $g_1,g_2\in \Sym(V)$ with $g_1\neq g_2$ such that \[(g_1,g_2)\in\Aut(BY)\setminus K.\]

Let $G\coloneqq \langle g_1,g_2\rangle$. By Lemma~\ref{lem:split}(\ref{lem:support at least 4}), we have that $|\supp(G)|\geq 4$. If $|\supp(G)|\geq 5$ , then by~\Cref{Lem:supportsize5}, the number of possibilities for $Y$ is at most $o\left(2^{\frac{n^2-5n+4}{2}}\right)$. We may thus assume that $|\supp(G)|=4$.
 We now consider the cases which occur in the proof of Corollary~\ref{COR: if suppG has order 4 then Y=X(A,B)}. In Case 1, 
there exists $A$ such that $Y=X(A,A)$ and it follows that $\Aut(Y)$ admits three distinct double transpositions (with support $V\setminus V(X)$). In Subcase 2.1, $\Aut(Y)$ contains the double transposition $g_1g_2$ and if $g_1g_2=pq$, then $(g_1,g_2)\in K$, contrary to the choice of $(g_1,g_2)$. Finally, Subcase 2.2 does not occur. Thus, in both cases that do occur, $\Aut(Y)$ admits two distinct double transpositions.
By Lemma \ref{lem:TwoDoubleTranspositions}, the number of such $Y$ is at most $o\left(2^{\frac{n^2-5n+4}{2}}\right)$. Hence, we obtain that
\[
|\CA_n|
=
o\left(2^{\frac{n^2-5n+4}{2}}\right),\]
as desired.
\end{proof}

\section*{Acknowledment}
The authors would like to thank to Rajko Nenadov for several helpful discussions regarding this problem.
Ademir Hujdurović gratefully acknowledges the hospitality of the University of Auckland and the award of a Kalman Visiting Scholarship in 2024, during which the research leading to this paper was initiated, as well as financial support from the Slovenian Research and Innovation Agency through research program P1-0404 and research projects N1-0481, N1-0428, N1-0459, J1-70035, J1-70046, J1-70047, and J1-60012.

\section*{Declaration of AI use}
AI was only used for minor polishing of the text.

\bibliography{References}{}
\bibliographystyle{plain}
\end{document}